\documentclass[conference]{IEEEtran}
\IEEEoverridecommandlockouts
\usepackage{cite}
\usepackage{amsmath,amssymb,amsfonts,dsfont,amsthm,cases}
\usepackage{algorithmic}
\usepackage{graphicx}
\usepackage{textcomp}
\usepackage{xcolor}

\newtheorem{theorem}{Theorem}

\newtheorem{corollary}{Corollary}

\def \QED {\hfill{$\Box$}}
\newenvironment{proofof}[1]{\noindent {\em Proof of #1.  }}{\QED}

\DeclareMathOperator*{\argmax}{arg\,max}
\DeclareMathOperator*{\argmin}{arg\,min}
\newcommand{\E}{\mathbb{E}}

\newcommand{\cut}[1]{}

\begin{document}

\title{Sequential Estimation of Network Cascades \thanks{This research was supported by the U.S. Air Force Office of Scientific Research under MURI Grant FA9550-18-1-0502.}
}

\author{\IEEEauthorblockN{Anirudh Sridhar}
\IEEEauthorblockA{\textit{Department of Electrical Engineering} \\
\textit{Princeton University}\\
Princeton, NJ \\
anirudhs@princeton.edu}
\and
\IEEEauthorblockN{H. Vincent Poor}
\IEEEauthorblockA{\textit{Department of Electrical Engineering} \\
\textit{Princeton University}\\
Princeton, NJ \\
poor@princeton.edu}
}

\maketitle

\begin{abstract}
We consider the problem of locating the source of a network cascade, given a noisy time-series of network data. Initially, the cascade starts with one unknown, affected vertex and spreads deterministically at each time step. The goal is to find an adaptive procedure that outputs an estimate for the source as fast as possible, subject to a bound on the estimation error. For a general class of graphs, we describe a family of matrix sequential probability ratio tests (MSPRTs) that are first-order asymptotically optimal up to a constant factor as the estimation error tends to zero. We apply our results to lattices and regular trees, and show that MSPRTs are asymptotically optimal for regular trees. We support our theoretical results with simulations. \\
\end{abstract}

\begin{IEEEkeywords}
Network cascade, sequential estimation, asymptotic optimality, hypothesis testing
\end{IEEEkeywords}

\section{Introduction}
Network cascades occur when the behavior of an individual or a small group of individuals diffuses rapidly through a network. Examples include the spread of epidemics in physical or geographical networks \cite{CF2010,flubreaks,Ant2015}, fake news in social networks \cite{tacchini2017some,AG17,Fourney2017}, and the propagation of viruses in computer networks \cite{KW91,MI17}. In each of these cases, the network cascade compromises the functionality of the network and it is of paramount importance to locate the source of the cascade as fast as possible.

This problem poses several interesting challenges. On one hand, network cascades are typically not directly observable even if one can monitor the network in real time. In the example of an epidemic spreading through a network, an individual's sickness could be caused by the epidemic or by exogenous factors (e.g., allergies). As the network is monitored over time, one may be able to distinguish between these possibilities at the cost of allowing the cascade to propagate further. Thus there is a fundamental tradeoff between the accuracy of the estimated cascade source and the amount of vertices affected by the cascade. How can we design algorithms that achieve the best possible tradeoff? 

In this paper, we take the first steps towards formalizing and solving the challenges addressed above. We begin by reviewing a model for network cascades with real-time noisy observations. We then study the problem of minimizing the expected run time of a sequential estimation algorithm for the cascade source subject to the estimation error being at most $\alpha$, for some $\alpha \in (0,1)$. We show that simple algorithms based on cumulative log likelihood ratios are first-order asymptotically optimal up to constant factors as we send the estimation error to zero for a large class of networks. In certain cases we can say more: the estimator we construct is first-order optimal in regular trees.

\subsection{A model of network cascades with noisy observations}
\label{subsec:cascade_model}

Let $G$ be a graph with vertex set $V$ and let time be indexed by an integer $t \ge 0$. We assume that the cascade starts from some vertex $v$ such that at $t = 0$, $v$ is the unique vertex affected. For any $t \ge 1$, a vertex $u$ is affected if and only if $u \in \mathcal{N}_v(t)$, where $\mathcal{N}_v(t)$ denotes the set of all vertices within distance $t$ from $v$ in the graph, with respect to the shortest path distance, denoted by $d(\cdot,\cdot)$. The cascade is not directly observable, but the system instead monitors {\it public states} $\{y_u(t) \}_{u \in V, t \ge 0}$.  Conditioned on the source being $v$, the public states are independent over all $u$ and $t$, with distributions given by 
$$
y_u(t) \sim \begin{cases}
Q_0 & u \notin \mathcal{N}_v(t); \\
Q_1 & u \in \mathcal{N}_v(t),
\end{cases}
$$
where $Q_0, Q_1$ are two distinct mutually absolutely continuous probability measures over $\mathbb{R}$. We can think of $y_u(t) \sim Q_0$ as typical behavior and $y_u(t) \sim Q_1$ as anomalous behavior caused by the cascade. This is a standard model which has been previously used, for instance, in studying quickest detection problems on networks \cite{ZouVeeravalli2018,ZVLT2018,ZVLT2019_IEEE,Rovatsos2019,qcd_empirical}.

\subsection{Formulation as a sequential hypothesis testing problem}
Let $(\Omega, \mathcal{F}, \mathbb{P})$ be a common probability space for all random objects. For each vertex $u$, let $H_u$ be the hypothesis that $u$ is the cascade source and let $\mathbb{P}_u : = \mathbb{P}(\cdot \mid H_u)$ be the associated measure. Any sequential estimator for the cascade source can be represented by a pair $(D,T)$, where $T$ is a (data dependent) stopping time and $D = \{D(t) \}_{t =0}^\infty$ is a sequence of estimators such that $D(t) \in V$ depends on the observations $y(0), \ldots, y(t)$. The output of the sequential estimation procedure is $D(T)$. Given a positive integer $R$ (which we will call the {\it confidence radius}), we say that the sequential estimator $(D,T)$ {\it succeeds} if $d(D(T),v) \le R$; else it fails. We can then formalize the tradeoff between estimator accuracy and number of infections as follows. Given $\alpha \in (0,1)$, we want to find the estimator which minimizes the worst-case expected runtime $\max_{v \in V} \E_v [ T]$, subject to the probability of failure being at most $\alpha$. 

A typical assumption in source estimation problems is that the graph $G$ has infinitely many vertices, is connected, and is locally finite\footnote{A graph is locally finite if the degree of each vertex is finite.} (see \cite{shah2011rumors}). However, the infinite graph setting corresponds to testing infinitely many hypotheses, and it is unclear whether there exists a sequential estimator with small error that will terminate in finite time. To remedy this situation, we will consider the behavior of sequential tests on finite restrictions of the graph. Formally, we fix a sequence $\{V_n\}_n$ where $V_n \subset V$ and $|V_n| = n$. In our analysis we will specify $n$ and assume that $D(t) \in V_n$; that is, the true source is an element of $V_n$. We will assume that $V_n$ is a neighborhood of a given vertex $v_0$ without loss of generality as the problem is only harder when all vertices are close to each other. 

We may now put these ideas together more formally to define a class of sequential estimators $\Delta_G(V_n, R,\alpha)$ for which the probability of failure is at most $\alpha$. Equations \eqref{eq:formulation_1} and \eqref{eq:formulation_2} below present two natural formulations for this set. 
\begin{align}
\label{eq:formulation_1}
\Delta_G' & : = \left \{ (D,T) : \forall v \in V_n, \mathbb{P}_v (d(D(T),v) > R) \le \alpha \right \} \\
\label{eq:formulation_2}
\Delta_G & : = \left \{ (D,T) :  \max\limits_{\substack{u,v \in V_n : \\ d(u,v) > R }} \mathbb{P}_v(D(T) = u) \le \frac{\alpha}{n} \right \}.
\end{align}
Formulation \eqref{eq:formulation_1} directly bounds the probability of failure, while formulation \eqref{eq:formulation_2} provides finer information on the distribution of the estimator outside of the confidence radius, and is thus more mathematically convenient. It is easy to see that $\Delta_G \subset \Delta_G'$. For any $(D,T) \in \Delta_G$ and for any $v \in V_n$, 
\begin{equation}
\label{eq:2_implies_1}
\mathbb{P}_v(d(D(T),v) > R) = \sum\limits_{u \in V_n \setminus \mathcal{N}_v(R)} \mathbb{P}_v(D(T) = u) \le \alpha. \hspace{-0.05cm}
\end{equation}
In other words, \eqref{eq:2_implies_1} shows that Formulation \eqref{eq:formulation_2} is stronger than Formulation \eqref{eq:formulation_1}. We believe that in certain cases, the two formulations are equivalent if $R$ is sufficiently large and $G$ satisfies some symmetry properties (e.g., vertex-transitivity). While we use \eqref{eq:formulation_2} in this paper, we will study the relationship between \eqref{eq:formulation_1} and \eqref{eq:formulation_2} in future work. 

Putting everything together, our goal will be to characterize
\begin{equation}
\label{eq:optimization}
T^*(V_n,R_n,\alpha) : = \min\limits_{(D,T) \in \Delta_G(V_n,R_n,\alpha)} \max\limits_{v \in V_n} \mathbb{E}_v \left[ T \right].
\end{equation}
For a fixed sequential estimator $(D,T)$, the inner maximum is the worst-case expected runtime of the estimator over all possible sources. The outer minimum is over all $(D,T)$ which meet the requirements of $\Delta_G(V_n,R_n,\alpha)$. In general sequential multi-hypothesis testing problems, characterizing the optimal test is intractable. We therefore study the asymptotics of \eqref{eq:optimization} first as $n \to \infty$ then as $\alpha \to 0$. We consider confidence radii $R_n$ that may be fixed with respect to all other parameters, or may grow with $n$. 

\subsection{Related work}

Although we are, to the best of our knowledge, the first to study this variant of the cascade source estimation problem, our work has close connections to several bodies of work. 

Shah and Zaman gave the first systematic study of estimating the source of a network cascade \cite{shah2011rumors,shah2010detecting}, which spawned several follow-up works, see for example \cite{shah2016finding,KhimLoh15,epidemic_message_passing,Luo2013IdentifyingIS,wang2014,feizi2015}. In their setup, they assume that the network cascade evolves according to a probabilistic model, and that at some future time a snapshot of the cascade is perfectly observed. The problems we address surrounding network cascades are complementary to this approach, and are more appropriate for the setting where one may monitor the state of the network in real time.  

Our work falls under the growing body of literature on sequential detection and estimation in networks. Recently Zou, Veeravalli, Li, Towsley and Rovatsos studied the problem of quickest detection of a network cascade \cite{ZouVeeravalli2018,ZVLT2018,ZVLT2019_IEEE,Rovatsos2019}, which is similar in nature to our work. The objective of their work is to detect with minimum delay when a cascade has started propagating in a network. They derive tests based on cumulative log-likelihood ratios that are shown to be asymptotically optimal when the growth rate of the cascade becomes very small. The detection problem with other cascade models were studied by Zhang, Yao, Xie and Qiu \cite{qcd_empirical}. Our work, on the other hand, studies problems of estimation. We assume simple cascade dynamics for ease of exposition, and extensions to other cascade dynamics is an important future direction. 

\section{Asymptotic behavior of optimal tests}
\label{sec:results}
At the core of the analysis is a characterization of the rate of convergence of the log-likelihood ratios. It is well known that tests based on log-likelihood ratios are optimal for distinguishing between two hypotheses \cite{wald1948} and are asymptotically optimal as the Type I error tends to 0 in the general multi-hypothesis testing problem \cite{Tartakovsky1998,VB1994,VB1995}. Thus, to motivate our results, we begin by studying some basic properties of the log-likelihood ratios that arise from our problem structure. 

For a positive integer $s$, define the shorthand $y(s) : = \{y_u(s)\}_{u \in V}$ to be the collection of public states at time $s$. We are interested in the cumulative log-likelihood ratio 
$$
Z_{vu}(t) : = \sum\limits_{s = 0}^t \log \frac{ d \mathbb{P}_v}{d \mathbb{P}_u} (y(s)). 
$$
From the cascade dynamics defined in Section \ref{subsec:cascade_model}, we can write the log-likelihood ratio $\log \frac{ d\mathbb{P}_v}{ d \mathbb{P}_u} (y(s))$ as
\begin{multline}
 \label{eq:likelihood_ratio_derivation}
 \log \frac{ \prod\limits_{w \in \mathcal{N}_v(s)} dQ_1( y_w(s)) \cdot \prod\limits_{w \notin \mathcal{N}_v(s)} dQ_0(y_w(s))}{ \prod\limits_{w \in \mathcal{N}_u(s)} dQ_1(y_w(s)) \cdot \prod\limits_{w \notin \mathcal{N}_u(s)} dQ_0(y_w(s))}  \\
 =  \hspace{-0.25cm} \sum\limits_{w \in \mathcal{N}_v(s) \setminus \mathcal{N}_u(s)} \hspace{-0.4cm} \log \frac{ dQ_1}{dQ_0}(y_w(s)) \hspace{0.1cm} + \hspace{-0.5cm} \sum\limits_{w \in \mathcal{N}_u(s) \setminus \mathcal{N}_v(s) } \hspace{-0.4cm} \log \frac{ dQ_0}{dQ_1}(y_w(s)).
\end{multline}
Under $\mathbb{P}_v$, all observed variables $y_w(s)$ are independent, with distributions given by 
\begin{equation}
\label{eq:observation_distributions}
y_w(s) \sim \begin{cases}
Q_0 & w \in \mathcal{N}_u(s) \setminus \mathcal{N}_v(s); \\
Q_1 & w \in \mathcal{N}_v(s) \setminus \mathcal{N}_u(s).
\end{cases}
\end{equation}
To simplify the analysis we assume that for each $u,v \in V$ and $t \ge 0$, $\mathcal{N}_v(t) \setminus \mathcal{N}_u(t)$ is nonempty, and $| \mathcal{N}_v(t) \setminus \mathcal{N}_u(t) | = | \mathcal{N}_u(t) \setminus \mathcal{N}_v(t) |$. This assumption holds for a large class of graphs (e.g., vertex-transitive graphs such as regular trees and lattices). For $u,v \in V$ define 
$$
f_{vu}(t) : = \sum\limits_{s = 0}^t | \mathcal{N}_v(s) \setminus \mathcal{N}_u(s) |.
$$
Equations \eqref{eq:likelihood_ratio_derivation} and \eqref{eq:observation_distributions} imply $\mathbb{E}_v \left[ Z_{vu}(t) \right] = \widetilde{D}(Q_0,Q_1) f_{vu}(t)$, where $\widetilde{D}(Q_0,Q_1)$ is the symmetrized Kullback-Leibler divergence between $Q_0$ and $Q_1$, given explicitly by 
$$
\widetilde{D}(Q_0,Q_1) := \int \log \left(\frac{d Q_1}{dQ_0} \right)  d Q_1 + \int \log \left(\frac{d Q_0}{d Q_1}\right)  d Q_0.
$$
Before stating our first main result on lower bounds for $T^*$, we will introduce some useful asymptotic notation. For two functions $g(n,\alpha)$ and $h(n,\alpha)$, we write, whenever it is well-defined, 
$$
g(n,\alpha) \gtrsim_{n,\alpha} h(n,\alpha) \iff \liminf\limits_{\alpha \to 0} \liminf\limits_{n \to \infty} \frac{ g(n,\alpha)}{ h(n,\alpha)} \ge 1. 
$$
The orderwise comparison operator $\lesssim_{n,\alpha}$ is analogously defined, but with limsups instead of liminfs. We also write $g(n,\alpha) \approx_{n,\alpha} h(n,\alpha)$ if and only if $g(n,\alpha) \gtrsim_{n,\alpha} h(n,\alpha)$ and $g(n,\alpha) \lesssim_{n,\alpha} h(n,\alpha)$. 

\begin{theorem}
\label{thm:lower_bound}
Let $F_{vu}$ be the inverse function of $f_{vu}$ and suppose $\log | \mathcal{N}_v(R_n) | \ll \log n$. Then 
\begin{equation*}
T^*(V_n,R_n,\alpha)  \gtrsim_{n,\alpha}  \max\limits_{\substack{u,v \in V_n : d(u,v) > 2R_n }} F_{vu} \left( \frac{\log n/ \alpha}{\widetilde{D}(Q_0,Q_1)} \right).
\end{equation*}
\end{theorem} 

\begin{proofof}{Theorem \ref{thm:lower_bound}}
We briefly discuss the high-level proof strategy. The term on the right hand side involving $F_{vu}$ is the time it takes for $\E_v [ Z_{vu}(t)]$ to cross the threshold $\log n/\alpha$, since $F_{vu}$ is the inverse function of the log-likelihood growth rate $f_{vu}$. Using a change of measure argument, we show $\E_v [ Z_{vu}(t)]$ {\it must} cross this threshold to achieve the guarantees of $\Delta_G(V_n, R_n, \alpha)$.

To show this formally, fix any $(D,T) \in \Delta_G(V_n,R_n,\alpha)$ as well as a vertex $v \in V_n$. Define the event $\Omega_{v,L} : = \{D(T) \in \mathcal{N}_v(R_n) \} \cap \{T \le L \}$, where $L$ is a positive integer to be chosen later. By a change of measure,  
$$
\mathbb{P}_u(D(T) \in \mathcal{N}_v(R_n)) = \mathbb{E}_v \left[ \mathds{1}_{\{D(T) \in \mathcal{N}_v(R_n)  \}} e^{ - Z_{vu}(T) } \right].
$$
For any positive integer $B$, 
\begin{multline*}
\mathbb{P}_u ( D(T) \in \mathcal{N}_v(R_n))  \ge \mathbb{E}_v \left[ \mathds{1}_{ \Omega_{v,L} \cap \{Z_{vu}(T) < B \} } e^{-Z_{vu}(T) } \right] \\
 \ge e^{-B} \mathbb{P}_v \left( \Omega_{v,L} \cap \left \{ \max\limits_{t \le L} Z_{vu}(t) < B \right \} \right) \\
 \ge e^{-B} \left( \mathbb{P}_v (\Omega_{v,L} ) - \mathbb{P}_v \left(\max\limits_{t \le L} Z_{vu}(t) \ge B \right) \right)
\end{multline*}
Noting that $\mathbb{P}_v(\Omega_{v,L}) \ge \mathbb{P}_v (D(T) \in \mathcal{N}_v(R_n) ) - \mathbb{P}_v(T > L)$ and substituting this in place of $\mathbb{P}_v(\Omega_{v,L})$ gives 
\begin{multline*}
\mathbb{P}_v(T > L) \ge \mathbb{P}_v(D(T) \in \mathcal{N}_v(R_n) ) \\
 - e^B \mathbb{P}_u(D(T) \in \mathcal{N}_v(R_n))  - \mathbb{P}_v \left( \max\limits_{t \le L} Z_{vu}(t) \ge B \right)
\end{multline*}
From \eqref{eq:formulation_2}, $\mathbb{P}_v (D(T) \in \mathcal{N}_v(R_n) ) \ge 1 - \alpha$ and $\mathbb{P}_u \left( D(T) \in \mathcal{N}_v(R_n) \right) \le \frac{  \alpha | \mathcal{N}_v(R_n ) | }{n}$ for $u \in V_n \setminus \mathcal{N}_v(2R_n)$. It follows that
\begin{multline}
\label{eq:lower_bound_expansion}
\mathbb{P}_v (T > L) \ge 1 - \alpha \\
 - e^B \cdot \frac{ \alpha | \mathcal{N}_v(R_n) |}{n } - \mathbb{P}_v \left( \max\limits_{t \le L} Z_{vu}(t) \ge B \right).
\end{multline}
Let $\epsilon > 0$, and set 
\begin{align*}
L &: = F_{vu} \left( \frac{1 - \epsilon}{ \widetilde{D}(Q_0,Q_1) + \epsilon} \log  \frac{ n}{ \alpha | \mathcal{N}_v(R_n)|}  \right) \\
B & : = (1 - \epsilon) \log \frac{ n}{ \alpha | \mathcal{N}_v(R_n)|}.
\end{align*}
Then by the strong law of large numbers (see Lemma 2.1 in \cite{Tartakovsky1998}), $\mathbb{P}_v \left( \max_{t \le L} Z_{vu}(t) \ge B \right)$ goes to 0 as $n \to \infty$. Plugging in these values to \eqref{eq:lower_bound_expansion} gives the following lower bound on $\mathbb{P}_v(T > L)$: 
\begin{equation*}
1 - \alpha - \left( \frac{\alpha | \mathcal{N}_v(R_n) |}{ n} \right)^{\epsilon} - \mathbb{P}_v \left(\max\limits_{t \le L} Z_{vu}(t) \ge B \right).
\end{equation*}
Take $n \to \infty$ and then $\alpha \to 0$ to obtain
$$
\liminf\limits_{\alpha \to 0} \liminf \limits_{n \to \infty} \mathbb{P}_v \left( T >  F_{vu} \left( \frac{(1 - \epsilon)\log \frac{ n}{\alpha | \mathcal{N}_v(R_n) |}}{ \widetilde{D}(Q_0, Q_1) + \epsilon}  \right) \right) \ge 1.
$$
We conclude from Markov's inequality and by considering only the first-order terms (see Lemma 2.1 in \cite{Tartakovsky1998}). 
\end{proofof}

Next we establish an upper bound on $T^*$ by considering families of matrix sequential probability ratio tests (MSPRTs) \cite{Tartakovsky1998,VB1994,VB1995}. Define the stopping time
$$
T_v : = \min\left\{ t \ge 0: \min\limits_{u \in V_n\setminus \mathcal{N}_v(R_n) } Z_{vu}(t) \ge \log \frac{n}{\alpha} \right \},
$$
and define the pair $(D_n,T_n)$ via 
\begin{align*}
T_n & : = \min_{v \in V_n} T_v, \hspace{0.5cm} D_n(t)  : = \argmax_{v \in V_n} \min_{u \in V_n \setminus \mathcal{N}_v(R_n)} Z_{vu}(t)
\end{align*}
so that $D_n(T_n) := \argmin_{v \in V_n} T_v$. It is simple to verify $(D_n,T_n) \in \Delta_G(V_n,R_n,\alpha)$. For $u \in V_n\setminus \mathcal{N}_v(R_n)$, 
\begin{multline*}
\hspace{-0.35cm}\mathbb{P}_v (D_n(T_n) = u)  \le \mathbb{P}_v \left( T_u < \infty \right) = \mathbb{E}_u \left[ \mathds{1}_{\{T_u < \infty \}} e^{- Z_{uv}(T_u) } \right] \\
 = e^{- \log n/\alpha} \mathbb{E}_u \left[ \mathds{1}_{ \{T_u < \infty \}} e^{- (Z_{uv}(T_u) - \log n/\alpha ) } \right].
\end{multline*}
Since $Z_{vu}(T_u) \ge \log n/\alpha$ by the definition of $T_u$, we have an upper bound of $\alpha/n$ as desired. The following theorem gives us an upper bound for the expectation of $T_n$. 

\begin{theorem}
\label{thm:upper_bound}
Let $\alpha \in (0,1)$ be fixed. There exists a constant $C(Q_0,Q_1) \in (0, \widetilde{D}(Q_0,Q_1))$ such that for every $v \in V_n$,
$$
\mathbb{E}_v \left[ T_n \right] \le \max\limits_{u \in V_n: d(u,v) > R_n} F_{vu} \left( \frac{ \log n }{ C(Q_0,Q_1) } \right) (1 + o_n(1)),
$$
where $o_n(1) \to 0$ as $n \to \infty$. 
\end{theorem}

\begin{proof}
We begin by upper bounding $\mathbb{P}_v(T_v > t)$. We can write 
\begin{align}
\mathbb{P}_v \left( T_v > t \right) & \le \mathbb{P}_v \left( \min\limits_{u \in V_n \setminus \mathcal{N}_v(R_n) } Z_{vu}(t) < \log \frac{n}{\alpha} \right) \nonumber \\
\label{eq:stopping_time_bound}
& \le \sum\limits_{u \in V_n \setminus \mathcal{N}_v(R_n) } \mathbb{P}_v \left( Z_{vu}(t) < \log \frac{n}{\alpha} \right).
\end{align}
Fix $\epsilon > 0$ and suppose that, for all $u \in V_n \setminus \mathcal{N}_v(R_n)$,  
$$
\log n/\alpha \le ( \widetilde{D}(Q_0,Q_1) - \epsilon) f_{vu}(t).
$$
To deal with the terms in the summation, we will use Chernoff bounds for $Z_{vu}(t)$. Suppose that $X,Y$ are independent random variables with $X \sim Q_1$ and $Y \sim Q_0$, and define  
$$
I(x) : = \sup\limits_{\lambda \ge 0} \left \{ \lambda x + \log \E \left[ \left( \frac{dQ_0}{dQ_1}(X) \right)^\lambda \left( \frac{dQ_1}{dQ_0}(Y) \right)^\lambda \right] \right \}. 
$$
Furthermore, we have the strict inequality $I(x) > 0$ for $x \le \widetilde{D}(Q_0,Q_1)$ since the $\lambda$-moments of $\frac{dQ_0}{dQ_1}(X)$ and $\frac{dQ_1}{dQ_0}(Y)$ exist for $\lambda \in [0,1]$. Since $Z_{vu}(t)$ can be written as a sum of i.i.d. random variables under $\mathbb{P}_v$, we have the bound
$$
\mathbb{P}_v (Z_{vu}(t) \le x f_{vu}(t) ) \le e^{- f_{vu}(t) I(x) } \text{ for } x \le \widetilde{D}(Q_0,Q_1).
$$
Hence we can bound the summation \eqref{eq:stopping_time_bound} by 
\begin{equation}
\label{eq:sum_bound}
\mathrm{exp}\left( \log n - \min_{u \in V_n \setminus \mathcal{N}_v(R_n)} f_{vu}(t)I(\widetilde{D}(Q_0,Q_1) - \epsilon) \right).
\end{equation}
Set $C(Q_0,Q_1) :  = \min \{ \widetilde{D}(Q_0,Q_1) - \epsilon, I(\widetilde{D}(Q_0,Q_1) - \epsilon) \}$ and define
$$
t_{n,\epsilon} : = \max\limits_{\substack{u \in V_n:d(u,v) > R_n}} F_{vu} \left( \frac{ \log n /\alpha}{C(Q_0,Q_1)} \right).
$$
Next, write $\mathbb{E}_v [ T_v ] = \sum_{t = 0}^\infty \mathbb{P}_v (T_v > t)$ and apply \eqref{eq:sum_bound} to bound $\E_v [ T_v]$ by 
$$
 t_{n,\epsilon} + \hspace{-0.25cm} \sum\limits_{t = t_{n,\epsilon} + 1}^\infty \hspace{-0.2cm}\mathrm{exp}\left(  \log n - \hspace{-0.4cm} \min\limits_{\substack{u \in V_n:  \\ d(u,v) > R_n } } f_{vu}(t) I(\widetilde{D}(Q_0,Q_1) - \epsilon) \right)
$$
which is $t_{n,\epsilon}(1 + o_n(1))$ where $o_n(1) \to 0$ as $n \to \infty$. The desired result follows from considering only first-order terms.
\end{proof}

\section{Applications to regular trees and lattices}
\label{sec:trees_and_lattices}

To apply Theorems \ref{thm:lower_bound} and \ref{thm:upper_bound}, it suffices to compute $F_{vu}$ for the graphs of interest. The following result shows that $(D_n,T_n)$ is asymptotically optimal as $n \to \infty$ and $\alpha \to 0$ in regular trees.

\begin{corollary}
\label{cor:trees}
Let $G$ be the infinite $k$-regular tree. If $R_n \ll \log n$, then the MSPRT is asymptotically optimal, and 
$$
T^*(V_n,R_n,\alpha) \approx_{n,\alpha} \frac{ \log \log n}{\log (k-1)}.
$$
\end{corollary}
\cut{
\begin{proof}
For distinct $u,v \in V$ we can write  
\begin{align}
\label{eq:neighborhood_diff_decomp}
f_{vu}(t) & = \sum\limits_{s = 0}^t \left ( | \mathcal{N}_v(s) | - | \mathcal{N}_v(s) \cap \mathcal{N}_u(s) | \right).
\end{align}
It suffices to compute the asymptotic behavior of $f_{vu}$ and $F_{vu}$ to apply Theorems \ref{thm:lower_bound} and \ref{thm:upper_bound}. From straightforward computations, $\sum_{s = 0}^t | \mathcal{N}_v(s) | \sim \frac{k^2}{(k-1)^2} \cdot k^{t}$. Next, we compute $\sum_{s = 0}^t | \mathcal{N}_v(s) \cap \mathcal{N}_u(s) |$. Without loss of generality, suppose that $d(u,v) = r$ is even. Then there is a unique vertex $w$ such that $d(w,v) = d(w,u) = r/2$, so $| \mathcal{N}_v(s) \cap \mathcal{N}_u(s) | = | \mathcal{N}_w(s - r/2) |$. It follows that 
$$
f_{vu}(t) \sim \frac{k^2}{(k-1)^2} \left( k^t - k^{t - r/2} \right) = \frac{k^2}{(k-1)^2} (1 - k^{-r/2} ) k^t.
$$
Up to first-order terms, $F_{vu}(z) \sim \frac{\log z}{\log k}$. The desired result follows from Theorems \ref{thm:lower_bound} and \ref{thm:upper_bound}.
\end{proof}
}

The idea behind the proof is that for $k \ge 3$, $f_{vu}(t) \asymp (k-1)^t$ and $F_{vu}(z) \sim \frac{\log z}{\log (k-1)}$. This follows from simple counting arguments so we do not include it here. In particular, the first-order asymptotics of $F_{vu}$ do not depend on $d(u,v)$; hence the first-order behavior of $T^*$ does not depend on $R$. 

To generate the numerical results in Figure \ref{fig:trees}, we let $G$ be a balanced, $k$-regular tree with height $h_k$ and root vertex $v_0$, where we set $h_3 = 15$ (32,767 vertices), $h_4 = 11$ (29,524 vertices) and $h_5 = 9$ (87,381 vertices). In each case we enumerated the vertices by positive integers so that if vertex $u$ is assigned a larger number than vertex $v$, then $d(u,v_0) \ge d(v,v_0)$. We set $V_n$ to be the set of vertices with label at most $n$, and allowed $n$ to range between 1,000 and 16,000. We set the root vertex $v_0$ to be the cascade source and also set $\alpha = 0.1$, $R = 0$, $Q_0 \equiv N(0,1)$ and $Q_1 \equiv N(2,1)$. An estimate of $\E_{v_0}[T_n]$ was obtained by averaging over 50 simulations per data point. 

\begin{figure}[t]
\centerline{\includegraphics[scale=0.4]{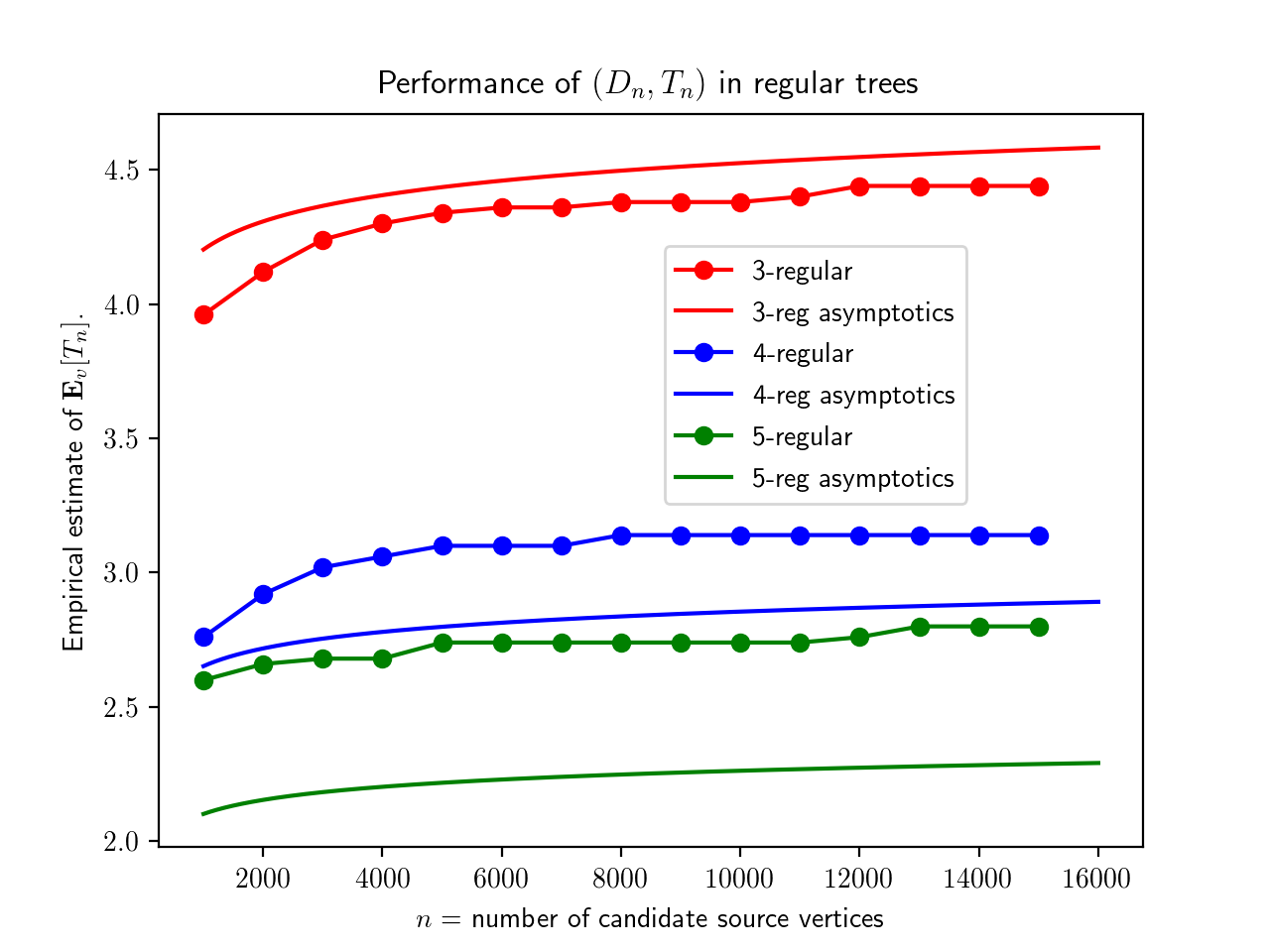}}
\caption{Numerical results for the MSPRT in regular trees for $R = 0$. }
\label{fig:trees}
\end{figure}

Next we consider optimal source estimation in lattices. We establish rigorous results in the one-dimensional case and discuss how things change in higher dimensions. Interestingly, the performance of the optimal estimator depends strongly on the confidence radius. 

\begin{corollary}
\label{cor:lattices}
Let $G$ be the infinite line graph. If $R_n \ll n$, then $(D_n,T_n)$ is asymptotically optimal up to a constant factor depending on $Q_0$ and $Q_1$. Let $C(Q_0,Q_1)$ be the constant from Theorem \ref{thm:upper_bound}. If $R_n \ll \sqrt{\log n}$, 
$$
\frac{\log n}{ 2R_n \widetilde{D}(Q_0,Q_1) } \lesssim_{n,\alpha} T^*(V_n,R_n,\alpha) \lesssim_{n,\alpha} \frac{\log n}{ R_n C(Q_0,Q_1) }.
$$
If $R_n \gg \sqrt{\log n}$ and $\log R_n  \ll \log n$, 
$$
\sqrt{ \frac{ \log n}{ \widetilde{D}(Q_0,Q_1) } } \lesssim_{n,\alpha} T^*(V_n,R_n,\alpha) \lesssim_{n,\alpha} \sqrt{ \frac{ \log n}{ C(Q_0,Q_1) }}.
$$
\end{corollary}

We sketch the derivation of $F_{vu}$. If $r = d(u,v)$ is even, then simple counting arguments show that 
\begin{equation}
\label{eq:f_vu_1d}
f_{vu}(t) = \begin{cases}
(t + 1)^2 & t < \frac{r}{2}; \\
\frac{r}{2} \left( 2t + 2 - \frac{r}{2} \right) & t \ge \frac{r}{2} .
\end{cases}
\end{equation}
Let $\lambda > 0$ be any constant. As $n \to \infty$, \eqref{eq:f_vu_1d} implies
\begin{equation}
F_{vu}(\lambda \log n) \sim \begin{cases}
\frac{\lambda \log n}{R_n} & R_n \ll \sqrt{\log n} \\
\sqrt{ \lambda \log n} & R_n \gg \sqrt{\log n}.
\end{cases}
\end{equation}
Applying Theorems \ref{thm:lower_bound} and \ref{thm:upper_bound} proves the corollary.

\cut{
\begin{proof}
It is easy to see that for any $v \in V$ and $s \ge 0$, $| \mathcal{N}_v(s) | = 2s + 1$. Next, fix $u,v \in V$ and assume without loss of generality that $d(u,v) = r$ is even. Thus $| \mathcal{N}_v(r/2) \cap \mathcal{N}_u(r/2) | = 1$, so for $s \ge r/2$, $| \mathcal{N}_v(s) \cap \mathcal{N}_u(s) | = | \mathcal{N}_w(s - r/2) | = 2s - r + 1$, where $w$ is the unique vertex in $\mathcal{N}_v(r/2) \cap \mathcal{N}_u(r/2)$. We can then compute
\begin{equation}
\label{eq:f_vu_1d}
f_{vu}(t) = \begin{cases}
(t + 1)^2 & t < \frac{r}{2}; \\
\frac{r}{2} \left( 2t + 2 - \frac{r}{2} \right) & t \ge \frac{r}{2} .
\end{cases}
\end{equation}
First suppose $R_n \ll (\log |V_n|)^{1/2}$. Then by \eqref{eq:f_vu_1d}, it holds for any constant $\lambda > 0$ that 
$$
\max\limits_{u,v \in V_n: d(u,v) > R_n } F_{vu}\left( \lambda \log |V_n|  \right) \sim \frac{\lambda \log |V_n|}{R_n }.
$$
Else if $R_n \gg (\log |V_n|)^{1/2}$, 
$$
\max\limits_{u,v \in V_n : d(u,v) > R_n} F_{vu} \left( \lambda \log |V_n| \right) \sim (\lambda \log |V_n|)^{1/2}.
$$
The desired result then follows from Theorems \ref{thm:lower_bound} and \ref{thm:upper_bound}.
\end{proof}
}

To generate the numerical results in Figure \ref{fig:radii} we let $G$ be a line graph with 1,000 vertices where we enumerate vertices from left to right from 1 to 1,000. We set $V_n$ to be the set of $n$ vertices closest to the vertex labelled 500 (where we chose odd values of $n$) and we let $n$ vary between 25 and 500. We studied the performance of the MSPRT when the cascade source is the vertex labelled 500 and set $\alpha = 0.2, Q_0 \equiv N(0,1)$ and $Q_1 \equiv N(0.5,1)$. To estimate $\E_{500}[T_n]$ we average each data point over 50 simulations. We plot our results for $R = 0, 5 \log n, \sqrt{n}$.  

Computing $f_{vu}$ and $F_{vu}$ in higher-dimensional lattices requires more involved combinatorial arguments, but Corollary \ref{cor:lattices} gives us a strong intuition of what to expect. The size of $|\mathcal{N}_v(s)|$ in the $k$-dimensional lattice is of order $s^k$, so for $t < \frac{d(u,v)}{2}$, $f_{vu}(t)$ should be on the order of $t^{k+1}$. Inverting $f_{vu}$, we expect that $T^*(n,R_n,\alpha)$ will increase as $(\log n)^{\frac{1}{k+1}}$ when $R_n \gg (\log n)^{\frac{1}{k+1}}$. 

\begin{figure}[t]
\centerline{\includegraphics[scale=0.4]{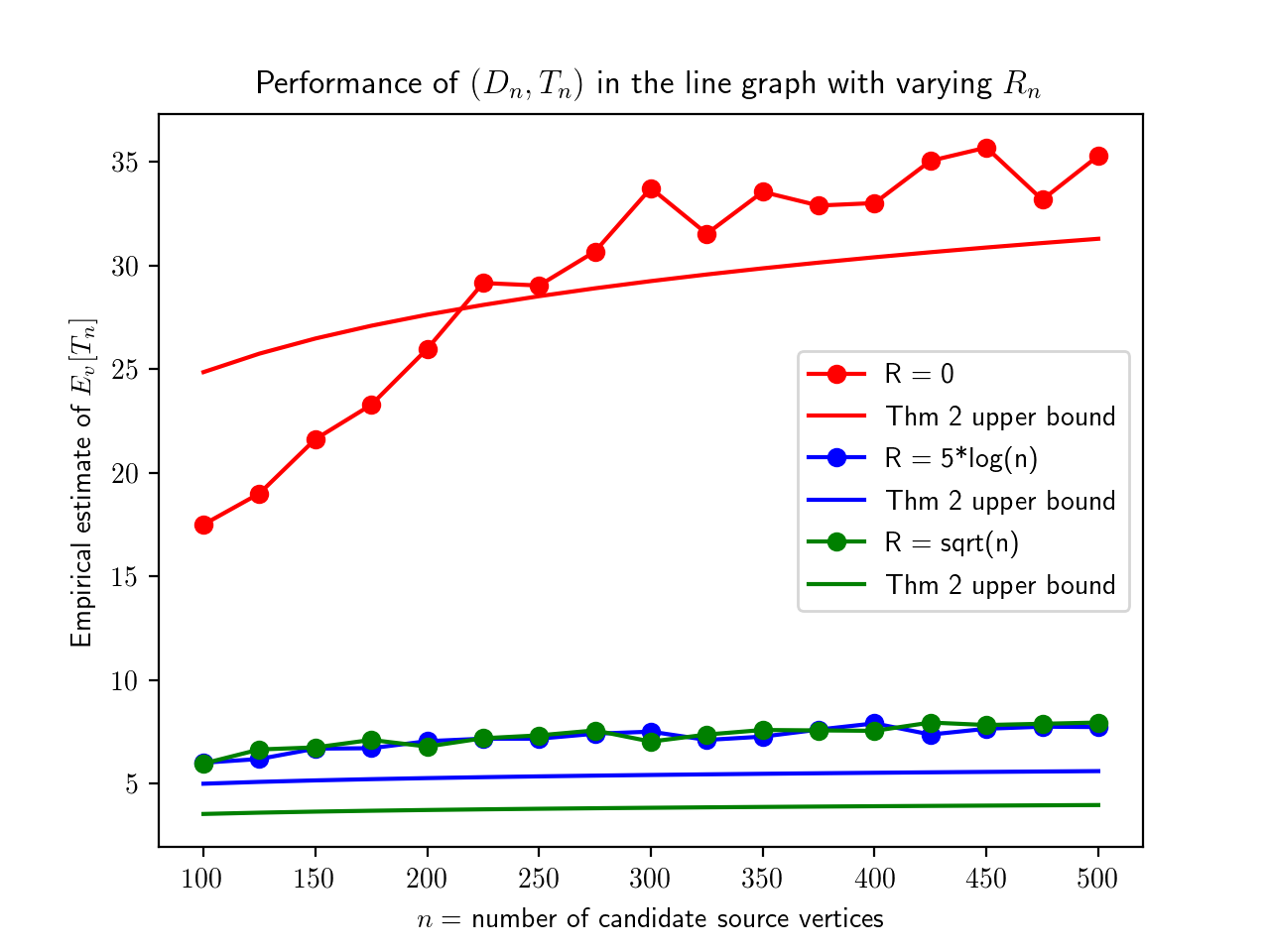}}
\caption{Numerical results for the MSPRT in the infinite line graph for $R_n = 0, 5 \log n, \sqrt{n}$. }
\label{fig:radii}
\end{figure}

\section{Conclusion}
\label{sec:conclusion}

In this paper, we studied the problem of estimating the source of a network cascade given noisy time-series data. We found that if $\min_{v \in V_n} |\mathcal{N}_v(R_n) | \ll n$, the MSPRT is asymptotically optimal as $\alpha \to 0$ in the case of regular trees, and is asymptotically optimal up to a constant factor in general. We have discussed many avenues for future work, including a study of {\it non-asymptotic} optimality, closing the gap between the upper and lower bounds of Theorems \ref{thm:lower_bound} and \ref{thm:upper_bound}, and investigating the relationship between the formulations \eqref{eq:formulation_1} and \eqref{eq:formulation_2} for $\Delta_G$.

\bibliographystyle{IEEEtran}
\bibliography{ciss_2020_final}

\end{document}